\documentclass[a4paper,11pt]{amsart}

\usepackage[left=2.3cm,right=2.3cm,top=3.5cm,bottom=3cm]{geometry}

\usepackage{amsmath,amssymb,amscd,amsfonts}
\usepackage[hyphenbreaks]{breakurl}
\usepackage{mathrsfs}
\usepackage{latexsym}
\usepackage[all]{xy}
\usepackage{verbatim}
\usepackage[usenames, dvipsnames]{color}
\usepackage{multirow}
\usepackage{url}
\usepackage{mathdots}
\usepackage{MnSymbol}
\usepackage{stmaryrd}
\allowdisplaybreaks

\usepackage{hyperref}
\hypersetup{colorlinks,breaklinks,
           linkcolor=MidnightBlue,urlcolor=MidnightBlue,
                       anchorcolor=MidnightBlue,citecolor=MidnightBlue}

\newcommand{\Z}{\mathbb{Z}}
\newcommand{\C}{\mathbb{C}}
\newcommand{\F}{\mathbb{F}}
\renewcommand{\P}{\mathbb{P}} %projective
\newcommand{\G}{\Gamma} %congruence subgroups
\newcommand{\T}{\mathbf{T}} %Hecke operator
\newcommand{\U}{\mathbf{U}} % Atkin operator
\newcommand{\W}{\mathbf{W}} % Atkin involution
\newcommand{\D}{\mathcal{D}}
\newcommand{\Tr}{\boldsymbol{Tr}} % trace
\newcommand{\g}{\gamma}
\renewcommand{\a}{\alpha}
\renewcommand{\b}{\beta}
\newcommand{\m}{\mathfrak{m}}
\newcommand{\n}{\mathfrak{n}}
\renewcommand{\d}{\mathfrak{d}}
\newcommand{\p}{\mathfrak{p}}

\newcommand{\mf}{\,|_{k,m}}

\newtheorem{thm}{Theorem}[section]
\newtheorem{prop}[thm]{Proposition}
\newtheorem{lem}[thm]{Lemma}
\newtheorem{cor}[thm]{Corollary}
\newtheorem{defin}[thm]{Definition}
\newtheorem{rem}[thm]{Remark}
\newtheorem{rems}[thm]{Remarks}

\renewcommand{\matrix}[4]{\left(\begin{array}{cc} {#1} & {#2} \\ {#3} & {#4} \end{array}\right)}
\newcommand{\smatrix}[4]{\left(\begin{smallmatrix} {#1} & {#2} \\ {#3} & {#4} \end{smallmatrix}\right)}

\title[Atkin--Lehner theory]{Atkin--Lehner theory for Drinfeld modular forms and applications}

\author{Maria Valentino}

\address{{\sc Maria Valentino}: Universit\`a della Calabria\\
Dipartimento di Matematica e Informatica\\
Ponte P. Bucci, Cubo 30B
87036 Rende (CS), Italy}

\email{maria.valentino@unical.it}

\begin{document}

\maketitle

\begin{abstract}
The present paper deals with Atkin--Lehner theory for Drinfeld modular forms.  We provide  an equivalent definition of $\p$-newforms (which makes computations easier)  and commutativity results between Hecke operators and Atkin--Lehner involutions.  As applications we show a criterion for a direct sum decomposition of cusp forms, we exibit $\p$-newforms
arising from lower levels and we provide $\p$-adic Drinfeld modular forms of level greater than 1. 
\end{abstract}

\medskip
\noindent {\bf MSC2020}: 11F52, 11F25.\\

\noindent {\bf Keywords}: {Drinfeld modular forms, Atkin--Lehner involutions, $\p$-newforms, $\p$-adic modular forms.}

\section{Introduction}
Let $p\in\Z$ be a prime, $q>1$ a $p$-power and fix integers $k>0$ and $m\in \{ 0,\cdots, q-2  \}$.
Let $S_{k,m}(\G_0(\n))$ be the space of Drinfeld cusp forms of weight $k$, type $m$ and level $\n$, where $\n=(\nu)$ is an ideal of $\F_q[t]$ with
monic generator $\nu\in \F_q[t]$.

When $\nu=t$, the structure of  $S_{k,m}(\G_0(t))$ has been investigated in the papers \cite{BV2} and \cite{BV3}. In particular, among other things, authors defined {\em oldforms}, i.e. cusp forms coming from the lower level
$GL_2(A)$, and {\em $(t)$-newforms} (see \cite[Section 3]{BV2}). Moreover, it was conjectured, and proved in some special cases, that  $S_{k,m}(\G_0(t))$ is direct
sum of oldforms and $(t)$-newforms and  a criterion for the direct sum decomposition was provided in \cite[Theorem 5.1]{BV3}. 

The above results were generalized in \cite{BV4}.  Specifically for the cases: 
\begin{itemize}
\item[\em a)] $\n$  is prime; 
\item[\em b)] $\n=\m\p$ where $\p$ is prime and $(\m,\p)=(1)$.
\end{itemize}
In the fist case, most formulas and definitions were straightforward generalization of the ones in \cite{BV2}. In the latter one, computations were more involved and
the criterion for the direct sum decomposition seemed to be only a sufficient condition.

In the present paper we shall exploit the so called {\em partial Atkin--Lehner involutions}
 (see Section \ref{AtLeInv} for more details) to twist trace maps and study the relations among these twists. As a consequence we achieve two main goals:
 \begin{itemize}
 \item[1.] {we supply an equivalent definition of $\p$-newforms of level $\m\p$ with $(\m,\p)=(1)$ that can fully generalize
 results in \cite{BV2} and \cite{BV3} and improve those in \cite{BV4};} 
 \item[2.] {we shed some light on the interaction between Hecke and Atkin--Lehner actions proving that they commute.}
 \end{itemize}

The paper is organized as follows.

In Section \ref{SecSttNot} we recall basic facts  and definitions on Drinfeld modular forms of
level $\n$, their Fourier expansions and Hecke operators $\T_\p$. In  Section \ref{AtLeInv} we provide a description of  Atkin--Lehner involutions $\W_\d^{\n}$, for $\d || \n$ (i.e. $\d$ divides $\n$ and $(\d,\frac{\n}{\d})=(1)$), and their main properties.

Section \ref{SecTwist} is devoted to the study of trace maps twisted by Atkin--Lehner involutions. We first recall from \cite{BV4} the 
definition of $\p$-oldforms. After that, we prove 

\begin{thm}[Theorem \ref{ThmComm}]
Let $\m,\p\subset  A$ be ideals  such that $\p\nmid \m$ and $\p$ is prime. Let $f\in S_{k,m}(\G_0(\m))$ and fix an ideal $\d$ of $A$ such that $\d ||\m$. Then
the Hecke operator $\T_\p$ commutes with the Atkin--Lehner involution $\W_\d^\m$.
\end{thm}

\noindent  Then, in Section \ref{SubSecTwist}, given the {\em trace map} $\Tr_\m^{\m\p}: M_{k,m}(\G_0(\m\p))\to M_{k,m}(\G_0(\m))$ we define 
 for any divisor $\d$ of $\m\p$ such that $\d|| \m\p$ the {\em $\d$-twisted trace map} (see Definition \ref{dtwistedTr})
 as
 \[   \Tr^{\m\p (\d)}_{\m}:= \Tr_\m^{\m\p}\circ \W_\d^{\m\p}:M_{k,m}(\G_0(\m\p))  \to M_{k,m}(\G_0(\m))\,,\]
 and we prove the following
  \begin{prop}[Proposition \ref{PropKerTraces}]
With notations as above, we have:
\begin{align*} \W^{\m\p}_\m\circ \Tr_\m^{\m\p(\d)}= \left\{ \begin{array}{ll} \delta^{2m-k} \Tr_\m^{\m\p (\frac{\m}{\d})} &  {\rm if}\ \p \nmid \d \\
\ & \\
(\frac{\delta}{P})^{2m-k} \Tr_\m^{\m\p (\frac{\m\p^2}{\d})} & {\rm if}\ \p | \d
\end{array} \right.\,.
\end{align*}

\end{prop}

We end Section \ref{SecTwist} providing our equivalent definition of $\p$-newform and proving the following

\begin{thm}[Proposition \ref{ThmCommNew} and Theorem \ref{DirSum}]
The involution $\W^{\m\p}_\p$ and the operator  $\U_\p$ commute on the space of $\p$-newforms of level
$\m\p$.\\
The map $\D:=Id-P^{k-2m}(\Tr_\m^{\m\p(\p)})^2$ is bijective on $S_{k,m}(\G_0(\m\p))$ if and only if we have the direct sum decomposition 
between $\p$-olforms and $\p$-newforms of level $\m\p$.
\end{thm}

We end the present paper with Section \ref{SecApp} in which we apply the theory of twisted traces to the following topics:
\begin{itemize}
\item[1.] {relation between $\p$-newforms of different levels (Lemma \ref{NewDiffLev});}
\item[2.] {$\p$-adic modular forms for the group $\G_0(\m)$ with $(\m,\p)=(1)$ and integral $u$-series coefficients (Proposition \ref{PropVinc}).}
\end{itemize}

\section{Setting and notations}\label{SecSttNot}

 Here we collect the main facts about Drinfeld modular forms and Atkin--Lehner involutions. 
 
 \subsection{Drinfeld modular forms}
 Let $K$ be the global function field $\F_q(t)$, where $q$ is a power of a fixed prime $p$, fix the prime  $\frac{1}{t}$ at $\infty$ and denote by 
 $A=\F_q[ t]$ its ring of integers. Let $K_\infty$ be  the completion of $K$ at $\infty$ and denote by $\C_\infty$ the completion of an algebraic closure of $K_\infty$.
The set $\Omega:=\P^1(\C_\infty)-\P^1(K_\infty)$, together with a structure of rigid analytic space (see \cite{FdP}), is the so called {\em Drinfeld upper half-plane}. 

The group $GL_2(K_\infty)$ acts on $\Omega$ via fractional linear transformations $$\matrix{a}{b}{c}{d}(z)=\dfrac{az+b}{cz+d}\,.$$ 
For $\g\in GL_2(K_\infty)$, $k,m\in\Z$ and $f: \Omega\to \C_\infty$,
we define the $\mf$ operator by 
\[  (f \mf \g)(z):=f (\g z)(\det \g)^m(cz+d)^{-k}. \]
Let us consider, for an ideal $\n$ in $A$,  the following arithmetic subgroup of $GL_2(A)$:
\[  \G_0(\n):=\left\{ \matrix{a}{b}{c}{d}\in GL_2(A)\ |\ c\equiv 0\pmod{\n}  \right\} \,. \]
The orbits for the action of $\G_0(\n)$ on $\P^1(K)$ are finitely many and are called the {\em cusps} of $\G_0(\n)$. 
The compactification of $\G_0(\n)\backslash\Omega$, obtained by adding to $\G_0(\n)\backslash \Omega$ the cusps of $\G_0(\n)$, is a connected,
smooth and proper curve over $\C_\infty$: it is denoted by $X_0(\n)$.

\begin{defin}
{\em A rigid analytic function $f:\Omega\to \C_\infty$ is called a} Drinfeld modular form of weight $k$ and type $m$ for $\G_0(\n)$ {\em if 
\[  (f\mf \g)(z)=f(z) \ \ \forall \g\in \G_0(\n)\] and it is holomorphic at all cusps.}\\
{\em If $f$ vanishes at all cusps, then it is called {\em cusp form}.}
\end{defin}

The space of Drinfeld modular forms of weight $k$ and type $m$ for $\G_0(\n)$ will be denoted by $M_{k,m}(\G_0(\n))$ and that of cusp forms by $S_{k,m}(\G_0(\n))$.

We recall that weight and type are not independent of each other: if $k\not\equiv 2m \pmod{q-1}$, then $M_{k,m}(\G_0(\n))=0$.

Both $M_{k,m}(\G_0(\n))$ and $S_{k,m}(\G_0(\n))$ are finite dimensional $\C_\infty$-vector spaces (see \cite{Ge80} for details on dimensions). Moreover, since
$M_{k,m}(\G_0(\n))\cdot M_{k',m'}(\G_0(\n))\subset M_{k+k',m+m'}(\G_0(\n))$ we have that

\[  M(\G_0(\n))=\bigoplus_{k,m} M_{k,m}(\G_0(\n))  \] is a graded $\C_\infty$-algebra.

 The space $M_{k,m}(\G_0(\n))$ is equipped with an Hecke action that we are going to describe. 
 
 \begin{defin}\label{DefHecke}
 {\em For $\p=(P)$ a prime ideal of $A$ with monic generator $P\in A$ of degree $d$, the} Hecke operator {\em acting on $f\in M_{k,m}(\G_0(\n))$ is 
 \begin{align*} 
  \T_\p(f)(z) & =  P^{k-m}\left[ (f\mf \matrix{P}{0}{0}{1})(z) + \sum_{\begin{subarray}{c}  Q\in A\\ \deg Q< d \end{subarray}}  (f\mf \matrix{1}{Q}{0}{P})(z)\right] \quad & {\mathrm{if}}\  \p\nmid \n \\
  \U_\p (f)(z)&  =P^{k-m} \sum_{\begin{subarray}{c}  Q\in A\\ \deg Q< d \end{subarray}}  (f\mf \matrix{1}{Q}{0}{P})(z)\quad & {\mathrm{if}}\ \p | \n\,.
 \end{align*}
 When $\p |\n$ the Hecke operator is also known as the} Atkin--Lehner {\em operator.}
 \end{defin}
 
 As observed in \cite[Section 7]{Ge88} Hecke operators preserve the space of cusp forms.\\
 
 Please note that the normalization used in the above definition differs from that of many other authors; we prefer keeping this one for coherence with formulas in
 previous papers on oldforms and newforms. 
 
 \subsubsection{Fourier expansions}
 Let $C$ be the {\em Carlitz module}, first studied in \cite{C}, and let $\tilde{\pi}\in K_\infty(\sqrt[q-1]{-t})$, defined up to a $(q-1)$th root of unity, be its {\em Carlitz period}.
 We fix such a $\tilde{\pi}$ for the reminder of this paper and let $L=\tilde{\pi}A$ be  the associated $A$-lattice.  The {\em Carlitz exponential} is
 the function $e_L(z): \C_\infty \to \C_\infty$
 \[  e_L(z):= z\prod_{\begin{subarray}{c} \lambda \in L \\ \lambda\neq 0\end{subarray}} \left(1-\frac{z}{\lambda}\right) \]
 which is entire and $\F_q$-linear in $z$ (more details and/or proofs can be found in \cite{Ge88}).
 
 Now consider the following analytic (rigid analytic) function on $\Omega$:
 \[  u:=u(z)=\frac{1}{e_L(\tilde{\pi}z)} ,  \]
which is invariant by translation $z\mapsto z+a $, $a\in A$ and therefore is a uniformizer at the cusp infinity.
 
 Let $f\in M_{k,m}(\G_0(\n))$. Then $f$ has a unique $u$-expansion 
 \[  f(z)=\sum_{i\geqslant 0}a_i u(z)^i \quad a_i\in \C_\infty  \]
 converging to $f$ for $z\in \Omega$ lying in a rigid analytic neighborhood of $\infty$ in the projective line.\\
 Observe that if $f$ is cuspidal then $a_0=0$.
 
 \subsubsection{Eisenstein series}\label{SecEis} We mention a few facts about a particular modular form of level 1 that we will need later.
 For $k\in \Z_{>0}$ and $z\in\Omega$, Goss defined in \cite{Go} an Eistenstein series 
 \begin{align*}  g_k:= (-1)^{k+1} \tilde{\pi}^{1-q^k} L_k \sum_{\begin{subarray}{c} a,b\in A \\ (a,b)\neq (0,0) \end{subarray}} \frac{1}{(az+b)^{q^k-1}} 
 \end{align*}
 where $L_k$ is the least common multiple of all monic polynomials of degree $k$. 
 It can be proved that $g_k$ is modular of weight $q^k-1$ and type 0 for $GL_2(A)$. Moreover, $g_k$ has integral, i.e in $A$, $u$-series coefficients (see \cite{Ge88}).
 
\subsubsection{$\p$-adic Drinfeld modular forms} Following Serre (\cite{Se}), Vincent  in \cite{Vi} defined $\p$-adic Drinfeld modular forms for $GL_2(A)$. 
Proceeding along the same lines, we shall generalize the definition to level $\G_0(\n)$.

As before, let $\p=(P)$ be a prime with $P\in A$ of degree $d$. We write $v_\p$ for the usual $\p$-adic valuation at $\p$. 

\begin{defin}
{\em Consider the formal power series $f=\sum_i a_i u^i\in K\llbracket u\rrbracket$. The} valuation of $f$ at $\p$ {\em is 
\[  v_\p(f)=\inf_i v_\p(a_i)\,. \]}
{\em We say that two formal series $f=\sum_i a_i u^i$ and $g=\sum_i b_i u^i$ in $K\llbracket u\rrbracket$ are} congruent modulo $\p^m$,
and we write  $f\equiv g$ $\pmod{\p^m}$, if {\em 
$v_\p(f-g)\geqslant m$.}
\end{defin}

\begin{defin}\label{DefpAdicSerre}
{\em A} $\p$-adic Drinfeld modular form of level $\n$ {\em is a formal power $u$-series expansion $f=\sum a_j u^j$
such that there exists a sequence $\{f_i\}$ of Drinfeld modular forms for $\G_0(\n)$ verifying  
$v_\p(f_i-f)\to\infty$  as $i\to \infty$.}
\end{defin}

\subsection{Atkin--Lehner involutions}\label{AtLeInv}
Here we define the  {\em Atkin--Lehner involutions}, summarize their fundamental properties and recall some basic results that will play a crucial role in what follows; for more details on such operators in the Drinfeld setting the reader is referred to the paper \cite{S}. 
We take the opportunity to remark that the behaviour of such involutions does not differ too much from the 
classical case, i.e. characteristic zero setting, and the results/proofs in this section follow (mutatis mutandis)  those in the original paper of Atkin and Lehner \cite{AL}. \\

Let $\n=(\nu), \d=(\delta) \subset A$ be ideals such that  $\d ||\n$. Denote by $W_\d^\n$ a matrix of the form 
\[ \matrix{\delta a}{b}{\nu c}{\delta d}\quad{\rm with}\quad   a,b,c,d\in A,\ \delta^2ad-\nu cb=\zeta\delta\ {\rm and}\   \zeta\in \F_q^*\,.\]

It is easy to verify that such matrices are in the normalizer of $\G_0(\n)$, i.e. 
\[ \matrix{\delta a}{b}{\nu c}{\delta d}\G_0(\n)\matrix{\delta a'}{b'}{\nu c'}{\delta d'}=\G_0(\n) \]
for any $\matrix{\delta a}{b}{\nu c}{\delta d}$ and $\matrix{\delta a'}{b'}{\nu c'}{\delta d'}$ as above. 
Therefore, we can give the following (well  posed) definition:

\begin{defin}
{\em Let $\n=(\nu), \d=(\delta) \subset A$ be ideals such that  $\d ||\n$.} The  (partial) Atkin--Lehner involution {\em  $\W_\d^\n$  acting on $M_{k,m}(\Gamma_0(\n))$  is:
\begin{align*}
\W_\d^\n : M_{k,m}(\Gamma_0(\n)) & \to M_{k,m}(\Gamma_0(\n))\\
f(z) & \mapsto (f\mf W_\d^\n) (z)\,
\end{align*}}
for any $W_\d^\n$ as above.
\end{defin}

In order to ease notations we shall drop the $(z)$ everywhere from now on. \\

It is easy to see that  the $\W_\d^\n$'s are involutions: indeed
\[  (\W_\d^{\n})^2(f)=(\zeta\delta)^{2m-k} f \,.\]

We also observe that,  for $\n_1=(\nu_1),\n_2=(\nu_2)$ with $(\n_i,\n/{\n_i})=(1)$, $i=1,2$, we can write

\[  W_{\n_1}^\n W_{\n_2}^\n=\smatrix{(\nu_1,\nu_2)}{0}{0}{(\nu_1,\nu_2)}W_{\frac{\n_1\n_2}{(\n_1,\n_2)^2}}^{\phantom{n}\n} \,. \]
This last formula also shows that the $W_\d^\n$'s commute.

Let $\n=\displaystyle{\prod_{i=1}^s}\d_i^{e_i}$ where $\d_i$ are distinct primes. Then $\displaystyle{\prod_{i=1}^s}W_{\d_i^{e_i}}^\n=W_\n^\n$ 
(except for multiplication by some $\zeta\in \F_q^*$) and 
it can be represented by the matrix $\matrix{0}{-1}{\nu}{0}$; if this is the case we obtain the  {\em (full) Atkin--Lehner involution} $W_\n^\n$, also
known as  {\em Fricke involution}.\\

The Atkin--Lehner involutions send cusps in cusps as the following theorem shows:

\begin{prop}\label{ActCusps}
With notations as above, we have:
\begin{itemize}
\item[{\em i)}] {every cusp of $X_0(\n)$ has a representative $x \choose y$ where $x,y\in A$ are monic, $y|\nu$ and $(x,\nu)=1$. Two representatives $x\choose y$, 
$x' \choose y'$ represent the same cusp if and only if $y=y'$ and $\zeta x'\equiv x \pmod{(y,y/\nu)}$ for some $\zeta\in \F_q^*$;}
\item[{\em ii)}] {if $x\choose y$ and ${x' \choose y'}= W_{\n_2}^\n {x \choose y}$ are cusps of $X_0(\n)$, where $\n=\n_1\n_2$, $\n_1=(\nu_1)$, $\n_2=(\nu_2)$ and $(\n_1,\n_2)=(1)$, then $y' =\frac{\nu_2}{(y,\nu_2)}(y,\nu_1)$.}
\end{itemize}
\end{prop}

\begin{proof}
See \cite[Proposition 1, Lemma 8]{S}.
\end{proof}

We end this section with the trivial remark that the Atkin--Lehner involutions preserve the space of cusp forms.

\section{Twisted trace maps}\label{SecTwist}
As already mentioned in previous papers on newforms, in the positive characteristic setting we do not have an analogue of the Petersson inner product. Therefore, following an idea of  
Serre (\cite[$\S$3 Remarque 3]{Se}), $\p$-newforms
in the Drinfeld case were defined using kernels of trace maps from $S_{k,m}(\G_0(\m\p))$ to $S_{k,m}(\G_0(\m))$.

This section is devoted to the study of trace maps twisted by Atkin-Lehner involutions and their applications to the study of spaces of newforms.

\subsection{$\p$-oldforms}\label{OldDef}
For the convenience of the reader we recall here the main results about  oldforms  from the paper \cite{BV4}. 
We also add some new results on the interaction between Hecke and Atkin-Lehner operators.\\

From now on, let $\n=\m\p$ for fixed ideals $\m=(\pi)$ and $\p=(P)$, where $\pi$ and $P$ are monic generators, $P$ is irreducible of degree $d$ and $(\pi,P)=1$.\\

Consider the spaces $M_{k,m}(\G_0(\m))$ and $M_{k,m}(\G_0(\m\p))$ and the {\em degeneracy maps}:
\begin{align*}
\mathbf{D}_1 : S_{k,m}(\G_0(\m)) & \to S_{k,m}(\G_0(\m\p))\\
 f & \mapsto f
\end{align*}
\begin{align*}
\mathbf{D}_\p : S_{k,m}(\G_0(\m)) & \to S_{k,m}(\G_0(\m\p))\\
 f  & \mapsto f\mf \matrix{P}{0}{0}{1}
\end{align*}

\begin{defin}
The space of $\p$-oldforms of level $\m\p$, denoted by  $S_{k,m}^{\p-old}(\G_0(\m\p))$, is the subspace of $S_{k,m}(\G_0(\m\p))$ generated by the set
$\{ (\mathbf{D}_1,\mathbf{D}_\p)(f_1,f_2)=\mathbf{D}_1 f_1+ \mathbf{D}_\p f_2: (f_1,f_2)\in S_{k,m}(\G_0(\m))^2  \}$.
\end{defin}

\begin{rem}
Since $\p\nmid \m$, the map $(\mathbf{D}_1,\mathbf{D}_\p): S_{k,m}(\G_0(\m))^2 \to S_{k,m}(\G_0(\m\p))$ is injective; the same map is not injective if $\p |\m$ (see \cite[Section 2]{BV4}). \\
We do not know if $(\mathbf{D}_1,\mathbf{D}_\p)$ is injective also on the space of modular forms. Indeed, the proof relies on the isomorphism
between the space of cusp forms and that of harmonic cocycles (\cite[Proposition 2.1]{BV4}). Therefore, even if most of the  involved maps and operators  can be defined for modular forms, when speaking of oldforms and newforms we shall only work with cusp forms. 
However all
results we need from \cite{BV4}, and that do not require the injectivity of $(\mathbf{D}_1,\mathbf{D}_\p)$, will be used for general modular
forms $f\in M_{k.m}(\G_0(\m\p))$ without any further comment. 
\end{rem}

We can rewrite the $\p$-oldforms in terms of the Atkin--Lehner involutions thanks to the following

\begin{lem}\label{DW}
Let $\n=(\nu), \d=(\delta) \subset A$ be ideals such that  $\d ||\n$. If $f\in M_{k,m}(\G_0(\d))$ then
\[  f\mf\matrix{\frac{\nu}{\delta}}{0}{0}{1}:=\mathbf{D}_{\frac{\n}{\d}} (f)= f\mf W_{\frac{\n}{\d}}^\n\,. \]
\end{lem}

\begin{proof}
Take $\a,\b\in A$ such that $\a \delta+\b \frac{\nu}{\delta}=1$, note that $\matrix{1}{-\a}{\delta}{\b \frac{\nu}{\delta}}\in\G_0(\d)$ and write
\[   f\mf \matrix{\frac{\nu}{\delta}}{0}{0}{1}= f\mf  \matrix{1}{-\a}{\delta}{\b \frac{\nu}{\delta}} \matrix{\frac{\nu}{\delta}}{0}{0}{1} = f\mf \matrix{\frac{\nu}{\delta}}{-\a}{\nu}{\b \frac{\nu}{\delta}} = f\mf W_{\frac{\n}{\d}}^\n\,. \qedhere\]
\end{proof}

Therefore 
\[  S_{k,m}^{\p-old}(\G_0(\m\p))=\mathrm{Span} \{ \W_1^{\m\p} (S_{k,m}(\G_0(\m))), \W_\p^{\m\p} (S_{k,m}(\G_0(\m)) ) \} \,. \]

\begin{rem}
The space of $\p$-oldforms is obviously a $\W_\p^{\m\p}$-invariant subspace of $S_{k,m}^{\p-old}(\G_0(\m\p))$.
\end{rem}

Our working definition of Hecke operators will make use of the above defined maps. Looking at Definition \ref{DefHecke}, for any $f\in M_{k,m}(\G_0(\m))$, we have
\[  \T_\p(f)=P^{k-m}\mathbf{D}_\p(f) + \U_\p(\mathbf{D}_1(f))\,. \]

By \cite[Proposition 2.13]{BV4}, we know that
\[ Ker(\U_\p)=Im (\mathbf{D}_\p)\,.  \] Therefore, $\U_\p$ preserves the space of $\p$-oldforms.

\begin{rem}\label{Integrality}
Before moving on we recall from \cite[$\S$ 3.1]{Vi} that if $f\in M_{k,m}(\G_0(\m))$ has $u$-expansion with integral coefficients, then the $u$-expansions of
$\U_\p(\mathbf{D}_1(f))$ and $\mathbf{D}_\p(f)$ have integral coefficients as well. In particular\footnote{It is easy to see that our different normalization does not affect these inequalities.}
\[  v_\p(\mathbf{D}_\p(f))\geqslant v_\p (f)\qquad \mathrm{and}\qquad  v_\p(\U_\p(f))\geqslant v_\p(f)\,. \]
\end{rem}

We prove next that some of the $\W_\d^\n$ commute with Hecke operators.

\begin{thm}\label{ThmComm}
With assumptions on $\m$ and $\p$ as above, let $f\in M_{k,m}(\G_0(\m))$ and fix an ideal $\d=(\delta)$ such that $\delta || \pi$. Then
\[  \W_\d^\m(\T_\p(f))=\T_\p(\W_\d^\m(f))\,. \]
\end{thm}

\begin{proof}
Since $\dfrac{\pi}{\delta}P^2$ and $\delta$ are coprime there exist $\a,\b\in A$ such that $\a \dfrac{\pi}{\delta}P^2 +\b \delta=1$. Then $\a\pi P^2+\b \delta^2=\delta$
and we put
\[   W_\d^\m=\matrix{\delta}{-\a}{\pi P^2}{\b \delta}. \]
We have that $\mathbf{D}_\p^{-1}\circ \W_\d^\m \circ \mathbf{D}_\p$ is represented by
\begin{align*}
\matrix{P}{0}{0}{1}\matrix{\delta}{-\a}{\pi P^2}{\b \delta}\matrix{\frac{1}{P}}{0}{0}{1} & =\matrix{\delta }{-\a P}{\pi P}{\b \delta}= W_\d^{\m}.
\end{align*} 
Therefore, we get  that 
\begin{equation}\label{CommD}
\mathbf{D}_\p \circ \W_\d^\m =\W_\d^\m\circ \mathbf{D}_\p\,.
\end{equation}
Observe now that, since $\b$ is coprime with $P$, for any $Q\in A$ with $\deg Q< d$ there exists a uniquely determined $Q'\in A$ with $\deg Q'< d$ such that $Q\equiv \b Q'\pmod{P}$. Moreover
\[  \matrix{1}{Q}{0}{P}\matrix{\b\delta}{-\a P}{\pi P}{\delta}=
\matrix{\delta\left(\beta +\frac{\pi}{\delta} P Q\right)}{\delta\frac{Q-\beta Q'}{P}-\pi QQ'-\a}{\pi P^2}{\delta\left(1-\frac{\pi}{\delta}PQ'\right) }
 \matrix{1}{Q'}{0}{P}\,. \]
Observe that the first matrix on the right has integral entries and is one of the representatives for $\W_\d^\m$. Therefore 
\begin{equation}\label{CommU}
\sum_{\begin{subarray}{c} Q\in A\\\deg Q< d  \end{subarray}}( f\mf \matrix{1}{Q}{0}{P}W_\d^\m)= 
\sum_{\begin{subarray}{c} Q'\in A\\\deg Q'< d  \end{subarray}}(f\mf W_\d^\m\matrix{1}{Q'}{0}{P})\,.
\end{equation}
Putting together equations \eqref{CommD} and \eqref{CommU} we get our claim.
\end{proof}

\begin{cor}\label{CorComm}
With assumptions as above on $\m$ and $\p$, if $f\in M_{k,m}(\G_0(\m\p))$  and $\d=(\delta)$ is such that $\delta || \pi$ we have
\[ \W_\d^{\m\p} (\U_\p(f) )=\U_\p(\W_\d^{\m\p}(f)) \,. \]
\end{cor}

\begin{proof}
Just observe that matrices used  in the previous proof for $\W_\d^\m$ are also appropriate for $\W_\d^{\m\p}$.
\end{proof}

\subsection{Twisted trace maps}\label{SubSecTwist}
Keep notations and assumptions on $\m$ and $\p$ as in the previous section, namely $(\m,\p)=(1)$, and let $R^{\m\p}_\m$ be a set of representatives for $\G_0(\m\p)\backslash \G_0(\m)$. 

 \begin{defin}
 {\em For a $f\in M_{k,m}(\G_0(\m\p))$ we define the}  trace map {\em
 as}
 \begin{align*}  \Tr^{\m\p}_\m:M_{k,m}(\G_0(\m\p)) & \to M_{k,m}(\G_0(\m))\\
 f &\mapsto \sum_{\g\in R^{\m\p}_\m}f\mf \g\,.
  \end{align*}
 \end{defin}
 
 As explained in \cite[Remark 2.9]{BV4} and originally observed by Serre in \cite{Se}, the trace alone is not enough to isolate newforms. As a consequence we give the following
 
 \begin{defin}\label{dtwistedTr}
 {\em For a $f\in M_{k,m}(\G_0(\m\p))$ and any divisor $\d$ of $\m\p$ such that $\d||\m\p$, we define  the}  $\d$-twisted trace map {\em
 as}
 \begin{align*}  \Tr^{\m\p (\d)}_{\m}:= \Tr_\m^{\m\p}\circ \W_\d^{\m\p}:M_{k,m}(\G_0(\m\p)) & \to M_{k,m}(\G_0(\m))\\
 f &\mapsto \sum_{\g\in R^{\m\p}_\m}(f\mf W_\d^{\m\p})\mf \g\,.
  \end{align*}
 \end{defin}
 
 Since in most of the paper we shall move between the levels $\m\p$ and $\m$, we shall simply write $\Tr^{(\d)}$ for $\Tr^{\m\p (\d)}_{\m}$ and use $\Tr$ for 
 $\Tr^{\m\p}_\m=\Tr^{\m\p(1)}_\m$, returning to the original notations when/if necessary. Moreover, whenever we write $\W_\d^{\m\p}$ we tacitly assume that it is well defined, i.e. 
 $\d |\m\p$  and  $(\d,\frac{\m\p}{\d})=(1)$.

\begin{lem}\label{LemTr}
Let $f\in M_{k,m}(\G_0(\m\p))$, then
\begin{equation}\label{eqTr} \Tr(f)=f+P^{-m}\U_\p(f\mf W_\p^{\m\p})\,. \end{equation}
\end{lem}

\begin{proof}
Let $\a,\b\in A$ verify $\a P+\b \pi=1$ and note that 
\[ \matrix{\a}{-\b}{\pi}{P}W_\m^\m=\matrix{-\b\pi}{-\a}{\pi P}{-\pi} \]
is a representative for $\W_\m^{\m\p}$. Then, 
by \cite[Eq. (12)]{BV4}, we have that
\[   \U_\p(f\mf W_{\m\p}^{\m\p})= P^m \Tr(f)\mf W_\m^\m-P^m f\mf W_\m^{\m\p}. \]
Since the involutions commute, using Corollary \ref{CorComm} we get
\begin{align*}
\U_\p(f\mf W_{\p}^{\m\p})\mf W_\m^{\m\p}= P^m \Tr(f)\mf W_\m^\m-P^m f\mf W_\m^{\m\p}.
\end{align*}
Apply $\W_\m^{\m\p}$ to both members to obtain
\[  \pi^{2m-k}\U_\p(f\mf W_{\p}^{\m\p})=P^m \Tr(f)\mf W_\m^\m W_\m^{\m\p}-\pi^{2m-k}P^m f\,.\]
Finally, with $\a,\b$ as above, we have
\[ W_\m^\m W_\m^{\m\p} =\matrix{0}{-1}{\pi}{0}\matrix{\pi}{-\a}{\pi P}{\b \pi}=\matrix{\pi}{0}{0}{\pi}\underbrace{\matrix{-P}{-\b}{\pi}{-\a}}_{\in \G_0(\m)} \,.\]
Since $\Tr(f)$ is $\G_0(\m)$-invariant, the lemma follows.
\end{proof}

\begin{rem}
An alternative proof of the above formula can be found  in \cite[Proposition 4.6]{DK}, but be aware of the different normalization used for the operator $\U_\p$. 
As in \cite{BV4}, they used a direct computation approach but with a different system of representatives: of course, the trace does not depend on the chosen 
$R_\m^{\m\p}.$ 
\end{rem}

By Lemma \ref{LemTr} we have:
\begin{align}\label{twist}  
\Tr^{(\d)}(f) :& =   \Tr(f\mf W_\d^{\m\p}) = f\mf W_\d^{\m\p} +P^{-m}\U_\p(f\mf W_\p^{\m\p} W_\d^{\m\p}) \nonumber \\
\ & =\left\{\begin{array}{ll}f\mf W_\d^{\m\p} +P^{-m}\U_\p(f\mf W_{\d\p}^{\m\p} ) & {\rm if}\ \p \nmid \d \\
\ & \\
f\mf W_\d^{\m\p} +P^{m-k}\U_\p(f\mf W_{\frac{\d}{\p}}^{\m\p} ) & {\rm if}\ \p | \d \end{array}\right. \ ,
 \end{align}
i.e., applying Corollary \ref{CorComm},
\begin{align}\label{twist2}
\Tr^{(\d)} = \left\{\begin{array}{ll} \W_\d^{\m\p} +P^{-m}\U_\p\circ \W_{\d\p}^{\m\p}=
 \W_\d^{\m\p} +P^{-m}\W_{\d}^{\m\p}\circ\U_\p\circ \W_\p^{\m\p} & {\rm if}\ \p \nmid \d \\
 \ & \\
\W_\d^{\m\p} +P^{m-k}\U_\p\circ \W_{\frac{\d}{\p}}^{\m\p}=
\W_\d^{\m\p} +P^{m-k}\W_{\frac{\d}{\p}}^{\m\p}\circ\U_\p & {\rm if}\ \p | \d \end{array}\right. \ ,
\end{align}

\begin{prop}\label{PropKerTraces}
With notations as above, we have:
\begin{align*} \W^{\m\p}_\m\circ \Tr^{(\d)}= \left\{ \begin{array}{ll} \delta^{2m-k} \Tr^{(\frac{\m}{\d})} &  {\rm if}\ \p \nmid \d \\
\ & \\
(\frac{\delta}{P})^{2m-k} \Tr^{(\frac{\m\p^2}{\d})} & {\rm if}\ \p | \d
\end{array} \right.\,.
\end{align*}
\end{prop}

\begin{proof} Just use \eqref{twist2}. \end{proof}

An immediate consequence of the above proposition is 
\begin{align}\label{EqKerTr}
Ker(\Tr^{(\d)}) = \left\{ \begin{array}{ll} Ker(\Tr^{(\frac{\m}{\d})}) & {\rm if}\ \p \nmid \d \\
\ & \\
 Ker(\Tr^{(\frac{\m\p^2}{\d})}) & {\rm if}\ \p | \d
\end{array} \right.\,.
\end{align}
In particular, $Ker(\Tr^{(\m\p)})=Ker(\Tr^{(\p)})$.

 \subsection{Drinfeld $\p$-newforms}
In \cite{BV4} the authors provide the following:

\begin{defin}\label{NewWrDef}
The space of $\p$-newforms of level $\m\p$, denoted by $S_{k,m}^{\p-new}(\G_0(\m\p))$,  is given by $Ker(\Tr)\cap Ker(\Tr^{(\m\p)})$.
\end{defin}

By equation \eqref{EqKerTr} it immediately follows that we have an equivalent definition for $\p$-newforms, namely
\[ S_{k,m}^{\p-new}(\G_0(\m\p))  = Ker(\Tr)\cap Ker(\Tr^{(\p)}) \] 
and this is a $W^{\m\p}_\p$-invariant subset of $S_{k,m}(\G_0(\m\p))$. As for the action of $W^{\m\p}_\p$ we can say more in the next proposition.

\begin{prop}\label{ThmCommNew}
The involution $\W^{\m\p}_\p$ and the operator  $\U_\p$ commute on the space of $\p$-newforms of level
$\m\p$.
\end{prop}

\begin{proof}
Let $f\in S_{k,m}^{\p-new}(\G_0(\m\p))$. By $\Tr^{(\p)}(f)=0$ and $\Tr(f)=0$ we get
\begin{equation*}  \U_\p(f) \mf W_\p^{\m\p}= -P^m f =  \U_\p(f\mf W_\p^{\m\p})\end{equation*}
and the claim follows easily. \end{proof}

Replacing $W_{\m\p}^{\m\p}$ with $W^{\m\p}_\p$, i.e. $\Tr^{(\m\p)}$ with $\Tr^{(\p)}$, makes formulas
easier to handle and  we can now obtain the full  generalization of \cite[Theorem 5.1]{BV3}.

\begin{thm}\label{DirSum}
The map $\D:=Id-P^{k-2m}(\Tr^{(\p)})^2$ is bijective on $S_{k,m}(\G_0(\m\p))$ if and only if we have the direct sum decomposition 
$S_{k,m}(\G_0(\m\p))=S_{k,m}^{\p-new}(\G_0(\m\p))\oplus S_{k,m}^{\p-old}(\G_0(\m\p))$.
\end{thm}

\begin{proof}
The argument is exactly as in \cite[Theorem 5.1]{BV3}, but we shall sketch out the main points to highlight the consequences of choosing another twist for the trace.

Assume that $\mathcal{D}$ is bijective and let $h=f + \mathbf{D}_{\p}(g)=f+\W_\p^{\m\p}g\in S_{k,m}(\G_0(\m\p))$, with $f,g \in S_{k,m}(\G_0(\m))$, be old and new. Therefore, 
$\Tr(h)=\Tr^{(\p)}(h)=0$ yield $f=-\Tr^{(\p)}(g)$ and $g-P^{k-2m}(\Tr^{(\p)})^2g =0$;  the bijectivity of $\mathcal{D}$ yields $g=f=0$, i.e. $S_{k,m}^{\p-new}(\G_0(\m\p))\cap S_{k,m}^{\p-old}(\G_0(\m\p))=0$.\\
For the sum condition, for any $h\in S_{k,m}(\G_0(\m\p))$ we look for $f,g \in S_{k,m}(\G_0(\m))$ such that $h-f-\mathbf{D}_\p (g)\in S_{k,m}^{\p-new}(\G_0(\m\p))$.
Working out the computations for $\Tr(h-f-\mathbf{D}_\p (g))=\Tr^{(\p)}(h-f-\mathbf{D}_\p (g))=0$, we find
\begin{itemize}
\item {$f=\mathcal{D}^{-1}(\Tr(h)-P^{k-2m}(\Tr^{(\p)})^2(h))$;}
\item {$g=\mathcal{D}^{-1}(P^{k-2m}(\Tr^{(\p)}(h)-(\Tr^{(\p)})^2(h))) $.}
\end{itemize}

Assume now that $S_{k,m}(\G_0(\m\p))=S_{k,m}^{\p-new}(\G_0(\m\p))\oplus S_{k,m}^{\p-old}(\G_0(\m\p))$ and let $f\in Ker(\mathcal{D})-\{0\}$. Then $f$ is in the image of a trace map, i.e. 
$f\in Im(\mathbf{D}_1)$ and $\U_\p(f)$ is old and non-zero. Moreover, noticing that $\Tr\circ \mathbf{D}_\p = P^{m-k}\T_\p$, we have
\begin{align*}
\Tr(\U_\p(f)) & = \Tr(\T_\p(f)-P^{k-m}\mathbf{D}_\p( f))\\
& = \T_\p(f)-P^{k-m} \Tr(\mathbf{D}_\p( f) )\\
& = \T_\p(f)-\T_\p(f)=0\,,
\end{align*}\ 
and 
\begin{align*}
\Tr^{(\p)}(\U_\p(f)) & = \Tr^{(\p)}(\T_\p(f)-P^{k-m}\mathbf{D}_\p (f))\\
 & = P^{k-m} \Tr^{(\p)}(\Tr(\mathbf{D}_\p( f)))-P^{k-m}\Tr^{(\p)}(\mathbf{D}_\p( f))\\
 & = P^{k-m}(\Tr^{(\p)})^2(f)-P^m f= -P^m \D(f)=0\,.
\end{align*}
We got that $\U_\p(f)$ is also new and therefore the contradiction to the direct sum.
\end{proof}

\begin{rem}
Comparing the calculations in the above proof with those in \cite[Theorem 2.18]{BV4}, we immediately observe that we finally get rid of $W_\m^\m$ ($Fr(\m)$ with notations as in \cite{BV4}) in all formulas. 

In order to understand why the twist by $W_\p^{\m\p}$ works better than the one by $W_{\m\p}^{\m\p}$ we have to look at  Proposition \ref{PropKerTraces}, which reflects  the Atkin--Lehner 
action on cusps. Indeed, by Proposition \ref{ActCusps}, $W_{\m\p}^{\m\p}$ switches the cusps $1 \choose 1$ and $1\choose \pi P$, while $W_\p^{\m\p}$ moves $1\choose 1$ to  $1 \choose  P$. This means that the full Atkin--Lehner involution sends the cusp $\infty$ to the cusp zero and when we go down in level $\m$ the $W_\m^\m$ remains involved
in all calculations. Unlike $W_{\m\p}^{\m\p}$, the action of the partial involution $W_\p^{\m\p}$ is only visible at level $\m\p$ and does not tamper with computations at level $\m$. 
\end{rem}

We end this section with one more remark about the relation between the eigenvectors of $\W_\p^{\m\p}$ and those of $\U_\p$.

\begin{lem}
Let $f\in S_{k,m}^{\p-new}(\G_0(\m\p))$. Then $f$ is an eigenvector for $\U_\p$ if and only if $f$ is an eigenvector for $\W_\p^{\m\p}$.
\end{lem}

\begin{proof}
It follows easily from equation \eqref{twist2} and simple direct computations. We just add that  the eigenvalues for $\W_\p^{\m\p}$ have opposite sign with respect to those of $\U_\p$ and have slope,
i.e. $\p$-adic valuation, $m-k/2$ (because the slope for $\U_\p$-eigenforms is $k/2$ by \cite[Theorem 2.16]{BV4}).
\end{proof}

\begin{rem}
We already observed in \cite[Remark 2.17]{BV4}  that $\U_\p$ is diagonalizable on the space of $\p$-newforms (if the characteristic is odd). Now we can add that $\U_\p$ and $\W_\p^{\m\p}$ can
be diagonalized using the same basis of eigenvectors. \\
We also wonder if $\U_\p$ is diagonalizible on $\p$-oldforms (again we are subject to the condition that $p$ is odd). In \cite[Remark 2.4]{BV4} it has been observed that $\U_\p$ is diagonalizable on $S_{k,m}^{\p-old}(\G_0(\m\p))$
if and only if $\T_\p$ is diagonalizable on $S_{k,m}(\G_0(\m))$ and is injective. It is easy to see that if $f\in S_{k,m}^{\p-old}(\G_0(\m\p))$ and $\W_\p^{\m\p}(f)=\lambda f$
then it must be $\lambda = \pm P^{m-k/2}$. Therefore, $f$ is also an eigenvector for $\U_\p$ if and only of it is $\p$-new. As a consequence, the problem of $\U_\p$ being
diagonalizable on the space of $\p$-oldforms remains open.
\end{rem}

 \section{Applications}\label{SecApp}
We collect in this last section some further results as applications of twisted trace maps. In the first part, we investigate the relation between $\p$-newforms at different levels. 
In the second and last part we shall address the topic of Drinfeld $\p$-adic forms \`a la Serre.\\

We keep notations as in previous sections: $\m=(\pi)$, $\p=(P)$ with $(\pi, P)=1$ and $P$ irreducible of degree $d$.

\subsection{Some oldforms as $\p$-newforms}

We begin by showing that $S_{k,m}^{\p-new}\!(\G_0(\p))\!\subseteq \!S_{k,m}^{\p-new}\!(\G_0(\m\p))$.\\
Since we will use the traces for different levels we will go back to the original notations. Moreover, we recall that
$S_{k,m}^{\p-new}(\G_0(\p))=Ker(\Tr_1^\p)\cap Ker(\Tr_1^{\p(\p)})$ (see \cite[Section 2.2]{BV4}).

\begin{lem}\label{NewDiffLev}
Let $f\in S_{k,m}(\G_0(\p))$ be a $\p$-newform of level $\p$. Then, $\mathbf{D}_1( f), \mathbf{D}_\m(f) \in S_{k,m}(\G_0(\m\p))$ are  $\p$-newforms of level $\m\p$.
\end{lem}

\begin{proof}
Since $f\in S_{k,m}^{\p-new}(\G_0(\p))$ we have that $\Tr^\p_1(f)=0$ and $\Tr^{\p(\p)}_1=\Tr^{\p}_1(f\mf W_\p^\p)=0$, namely
\[  f=-P^{-m}\U_\p(f\mf W_\p^\p)\quad \mathrm{and}\quad f\mf W_\p^\p = -P^{m-k}\U_\p(f)\,. \]
Fix $\a,\b\in A$ such that $\a P+\b \pi=1$ and take now $\mathbf{D}_1( f)\in S_{k,m}(\G_0(\m\p))$. We have
\begin{align*}
\Tr^{\m\p}_\m(f) & = f+P^{-m}\U_\p(f\mf W_\p^{\m\p})\\
 & = f+ P^{-m}\U_\p(f\mf \matrix{P}{-\b}{\pi P}{\a P})\\
 & = f+ P^{-m}\U_\p(f\mf \matrix{\b}{1}{-\a P}{\pi} \matrix{0}{-1}{P}{0})\\
 & = f+P^{-m}\U_\p(f\mf W_\p^\p) =0
\end{align*}
(note, in particular, the relation $f\mf W^{\m\p}_\p=f\mf W^\p_\p$ for any $f\in S_{k,m}(\G_0(\p))$).
 Moreover,  
\begin{align*}
\Tr^{\m\p(\p)}_\m (f) = \Tr^{\m\p}_\m(f\mf W_\p^{\m\p})& = f\mf W_\p^{\m\p} +P^{m-k}\U_\p(f)\\
 & = f\mf W^\p_\p +P^{m-k}\U_\p(f)=0\,.
\end{align*}

As for $\mathbf{D}_\m( f)\in S_{k,m}(\G_0(\m\p))$ we first observe that thanks to Lemma \ref{DW}:
\begin{align*}
\mathbf{D}_\m( f) = f\mf \matrix{\b}{-\a}{P}{\pi} \matrix{\pi}{0}{0}{1} =f\mf W_\m^{\m\p}.
\end{align*}
Finally we have
\begin{align*}
\Tr^{\m\p}_\m (\mathbf{D}_\m( f)) & = \mathbf{D}_\m (f) +P^{-m} \U_\p (\mathbf{D}_\m (f)\mf W_\p^{\m\p})\\
 & =  f\mf W_\m^{\m\p} +P^{-m} \U_\p (f\mf W_{\m\p}^{\m\p}) \\
 & =\W^{\m\p}_\m (f+P^{-m} \U_\p (f\mf W^{\m\p}_\p)) \qquad{\rm (Lemma\ \ref{CorComm})} \\
 & =\W^{\m\p}_\m (f+P^{-m} \U_\p (f\mf W_\p^\p)) = 0; 
\end{align*}
\begin{align*}
\Tr_\m^{\m\p(\p)} (\mathbf{D}_\m (f))& =\Tr^{\m\p}_\m(\mathbf{D}_\m (f)\mf W_\p^{\m\p})\\ & = \mathbf{D}_\m( f)\mf W_\p^{\m\p} +P^{m-k} \U_\p(\mathbf{D}_\m(f))\\
&  = f\mf W^{\m\p}_{\m\p} +P^{m-k}\U_\p(f\mf W^{\m\p}_\m) \\
&  = \W^{\m\p}_\m(f\mf W^{\m\p}_\p + P^{m-k}\U_\p(f)) \qquad{\rm (Lemma\ \ref{CorComm})} \\
&  = \W^{\m\p}_\m(f\mf W^\p_\p + P^{m-k}\U_\p(f))=0. \qedhere  
\end{align*}
\end{proof}

We guess that, for dimensional reasons, at level $\m\p$ the are {\em genuine} $\p$-newforms that cannot arise from any other underlying level. It would be interesting to  further  investigate the subject.

\subsection{$\p$-adic Drinfeld modular forms}

In the paper \cite{Vi}, C. Vincent extends the notion of $p$-adic modular form introduced by Serre (see \cite{Se}) to the 
 Drinfeld setting by 
showing that Drinfeld modular forms for the group $\G_0(\p)\subseteq GL_2(\F_q[t])$ with integral $u$-expansions are also $\p$-adic modular forms. 
Here we just observe that thanks to Lemma \ref{LemTr} the same result follows for the groups $\G_0(\m\p)\subseteq \G_0(\m)$.

\begin{prop}\label{PropVinc}
Let $f\in M_{k,m}(\G_0(\m\p))$ with rational $u$-series coefficients, where $(\m,\p)=(1)$ and $\p$ is prime. Then $f$ is a $\p$-adic Drinfeld modular forms for $\G_0(\m)$. 
\end{prop}

\begin{proof}
The proof works exactly as in \cite[Theorem 4.1]{Vi}. Therefore, we are just going to briefly outline the main points because of the different normalization used in this paper.\\
For a positive integer $n$, let $$g_{(0)}:= P^{n(q^d-1)}g_d^n-P^{2n(q^d-1)}g_d^n |_{n(q^d-1),0} W_\p^{\m\p}\in M_{n(q^d-1),0}(\G_0(\m\p)),$$ where $g_d$ is the 
Eistenstein series of weight $q^d-1$ and type 0 for $GL_2(A)$ introduced in Section \ref{SecEis} \footnote{Note that according to our convenience or purpose  $g_d=\mathbf{D}_1(g_d)$ can be seen as a modular form of level $\m$, so that $\mathbf{D}_\p(g_d)=g_d|_{q^d-1,0} W_\p^{\m\p}$ is of level $\m\p$, or directly as a form of level
$\m\p$.}.  \\
Using $g_d\equiv 1 \pmod{\p}$ (see \cite{Ge88}), Remark \ref{Integrality} and the fact that $\mathbf{D}_\p (g_d^n)=\mathbf{D}_\p(g_d^n-1)+1\equiv 1\pmod{\p}$, we get 
\[  g_{(0)} |_{n(q^d-1),0} W_\p^{\m\p} \equiv 0 \pmod{\p^{n(q^d-1)+1}}\,.\]
We now define $g_{(r)}:=g_{(0)}^{p^r}$. In the same way we obtain
\[  g_{(r)}|_{p^rn(q^d-1),0}W_\p^{\m\p}\equiv 0\pmod{\p^{np^r(q^d-1)+p^r}}\,.  \]
Consider $fg_{(r)}\in M_{k+np^r(q^d-1),m}(\G_0(\m\p))$.  Thus, $\Tr(fg_{(r)})\in M_{k+np^r(q^d-1),m}(\G_0(\m))$ and
\[ \Tr(fg_{(r)})-f = \underbrace{\Tr(fg_{(r)})- fg_{(r)}}_{A} + \underbrace{fg_{(r)}-f}_{B}\,.  \]
For the $B$ part we have 
\[  v_\p(f(g_{(r)}-1))\geqslant v_\p(f)+p^r\,, \]
while for the $A$ part (recall formula \eqref{eqTr} and Remark \ref{Integrality}) we get
\[ v_\p(\Tr(fg_{(r)})- fg_{(r)})\geqslant  p^r+p^rn(q^d-1)-m +v_\p(f\mf W_\p^{\m\p})\,. \]
Since 
\[ v_\p(\Tr(fg_{(r)})-f) \geqslant \min \{ p^r+v_\p(f) , p^r+p^rn(q^d-1)-m +v_\p(f\mf W_\p^{\m\p})\}\,, \]
we get our claim. Indeed, this shows that $f$ is a $\p$-adic modular form of level $\G_0(\m)$, because the sequence $\{\Tr(fg_{(r)})\}$ satisfies the requirements of Definition \ref{DefpAdicSerre}.
\end{proof}

\begin{rems}\begin{enumerate}
\item {A similar argument was used in \cite{DK} to prove interesting $\mathrm{mod}\ {\p}$ congruences between cusp form of different levels.}
\item {For the sake of completeness, we have also to mention that  more relevant progresses in the construction of families of Drinfeld modular forms can be found in the works of S. Hattori (see \cite{H}) and M.-H. Nicole and G. Rosso (\cite{NR1} and \cite{NR2}).}
\end{enumerate}
\end{rems}


\begin{thebibliography}{99}

\bibitem{AL} {\sc A.O. L. Atkin, J. Lehner} {\em Hecke operators on $\G_0(m)$}, Math. Ann. {\bf 185}, (1970) 134--160.

\bibitem{BV2} {\sc A. Bandini, M. Valentino} {\em  On the Atkin $U_t$-operator for $\Gamma_0(t)$-invariant Drinfeld cusp forms}, 
Proc. Amer. Math. Soc. {\bf 147} (2019), no. 10, 4171--4187.

\bibitem{BV3} {\sc A. Bandini, M. Valentino} {\em On the structure and slopes of Drinfeld cusp forms}, 
to appear in {\em Exp. Math.} \href{https://doi.org/10.1080/10586458.2019.1671921}{https://doi.org/10.1080/10586458.2019.1671921}.

\bibitem{BV4} {\sc A. Bandini, M. Valentino}  {\em Drinfeld cusp forms: oldforms and newforms}, to appear in J. Number Theory \\
\href{https://doi.org/10.1016/j.jnt.2020.03.011}{https://doi.org/10.1016/j.jnt.2020.03.011}.

\bibitem{C} {\sc L. Carlitz} {\em An analogue of the von Staudt-Clausen theorem}, Duke Math. J. {\bf 3} (1937), 503--517.

\bibitem{DK} {\sc T. Dalal, N. Kumar} {\em On mod $\p$ congruences for Drinfeld modular forms of level $\p\m$},  arXiv:2008.01959 [math.NT] (2020).

\bibitem{FdP} {\sc J. Fresnel, M. van der Put} {\em G\'eom\'etrie Analytique Rigide et Applications}, Progress in Mathematics,
vol. {\bf 18}, Birkh\"auser, 1981.

\bibitem{Ge80} {\sc E.-U. Gekeler} {\em Drinfeld Modular curves}, Lecture Notes in Mathematics {\bf 1231} (Springer-Verlag, 1980).

\bibitem{Ge88} {\sc E.-U. Gekeler} {\em On the coefficients of Drinfeld modular forms}, Invent. Math. {\bf 93} (1988), 667–700.

\bibitem{Go} {\sc D. Goss} {\em $\pi$-adic Eisenstein series for function fields}, Compos. Math. {\bf 41} (1980) 3--38.

\bibitem{H} {\sc S. Hattori} {\em $\wp$-adic continuous families of Drinfeld Eigenforms of finite slope}, arXiv:1904.08618 [math.NT], 2019.

\bibitem{NR1} {\sc M.-H. Nicole, G. Rosso} {\em Familles de formes modulaires de Drinfeld pour le groupe g\'en\'eral lin\'eaire}, arXiv:1805.08793 [math.NT], 2018.

\bibitem{NR2} {\sc M.-H. Nicole, G. Rosso} {\em Perfectoid Drinfeld modular forms}, arXiv:1912.07738 [math.NT], 2019.

\bibitem{S} {\sc A. Schweizer} {\em Hyperelliptic Drinfeld modular curves}, { Drinfeld modules, modular schemes and applications}, (Alden-Biesen, 1996), 330--343, 
World Scientific Publishing Co., Inc., River Edge, NJ, 1997.

\bibitem{Se} {\sc J.-P. Serre} {\em Formes modulaires et fonctions zeta p-adiques}, Modular forms in one variable III, 191--268, Lecture
Notes in Mathematics, vol. {\bf 350}, Springer Verlag, 1973.

\bibitem{Vi} {\sc C. Vincent} {\em On the trace and norm maps from $\Gamma_0(\mathfrak{m})$ to $GL_2(A)$},
J. Number Theory {\bf 142} (2014), 18--43.

\end{thebibliography}
\end{document}